\documentclass[12pt]{amsart}
\usepackage{amsfonts}

\usepackage[letterpaper]{geometry}

\usepackage{amssymb}
\def\Z{\mathbb{Z}}
\def\N{\mathbb{N}}

\def\bo{\{0,1\}^n}

\newtheorem{thm}{\bf Theorem}[section]
\newtheorem{lemma}[thm]{\bf Lemma}

\theoremstyle{definition}

\begin{document}

\title{Some remarks on the distribution of additive energy}

 \author[Norbert Hegyv\'ari]{Norbert Hegyv\'ari}
 \address{Norbert Hegyv\'{a}ri, ELTE TTK,
E\"otv\"os University, Institute of Mathematics, H-1117
P\'{a}zm\'{a}ny st. 1/c, Budapest, Hungary and associated member of Alfr\'ed R\'enyi Institute of Mathematics, Hungarian Academy of Science, H-1364 Budapest, P.O.Box 127.}
 \email{hegyvari@renyi.hu}

\begin{abstract}
The aim of this note is two-fold.  In the first part of the paper we are going to investigate an inverse problem related to additive energy. In the second, we investigate how dense a subset of a finite structure can be for a given additive energy. AMS 2010 Primary 11B75, Secondary 11B13

\end{abstract}

 \maketitle

\section{Introduction}

The additive energy is a central notion in additive combinatorics. This concept was introduced by Terence Tao and has been the subject of many works (See, e.g., \cite{S} and \cite{TV}.) For a set $A\subseteq \N$, the {\it additive energy} of $A$ is defined as the number of quadruples  $(a_1,a_2,a_3,a_4)$ for which $a_1+a_2=a_3+a_4$, formally $E(A):=|\{(a_1,a_2,a_3,a_4)\in A^4: \ a_1+a_2=a_3+a_4\}|$. The additive energy is similarly defined on an arbitrary structure $X$, where the addition is defined.

Let us remark that if $A$ is a finite subset of the integers then for every $x\in \Z$ $d_A(x):=|\{(a_1,a_2)\in A^2: \ 0\neq x=a_2-a_1 \}|=2d^+_A(x):=|\{(a_1,a_2)\in A^2:\ a_1< a_2; \ x=a_2-a_1; \}|$ (indeed for every $x$, $d_A(x)=d_A(-x)$). If from the content is clear, we leave the subscript and we write simply $d(x)$ or $d^+(x)$. Furthermore (perhaps this is a folklore) $\max_{A\in \N; \ |A|=n}E(A)=n^2+\frac{(n-1)n(2n-1)}{3}=(1+o(1))\frac{2}{3}n^3$. Indeed for every $x$ where $d(x)>0$ let $i$ be the maximal index for which $x=a_{i+1}-a_j; \ j\leq i$. Then $d(x)\leq i$. This attends when $A$ is an arithmetic progression with length $n$. A simply calculation shows the bound above.

Clearly $|A|^2\ll E(A)\ll |A|^3$ holds, since the quadruple $(a_1,a_2,a_1,a_2)$ is always a solution and given $a_1,a_2,a_3$ the term $a_4$ is uniquely determined by them. (We will use the notation $|X|\ll |Y|$ to denote the estimate $|X|\leq C|Y|$ for some absolute constant $C>0$). For every $M\geq 2$ we write $[M]:=\{1,2,\dots,M\}$.

The aim of this note is two-fold.  In the first part of the paper we are going to investigate an inverse problem related to additive energy. In the second part, we investigate how dense a subset of a finite structure can be for a given additive energy.

\section{An inverse problem}

There are various type of inverse problems. Maybe the best known is the celebrated Freiman-Ruzsa result which describes the structure of sets with small doubling $A+A; \ A\subseteq \N$ (see e.g. \cite{TV}). 

An other problem which is due to S. Burr asked which property of an infinite sequence $B$ ensures that $\N\setminus B$ can be written as a subset sum of an admissible sequence $A$, i.e. $\N\setminus B=P(A)=\{\sum_{a\in A'}a: A'\subseteq A; \ |A|<\infty\}$. This issue has a relatively large literature too (see. e.g. \cite{B}, \cite{CH}). We mention two other inverse problems which relate to the question of the present section. The first is originated from the folklore; it is known that for every finite set of integers $A$, $2|A|-1\leq |A+A|\leq {|A|+1 \choose 2}$. Is it true that for every $n,k\in \N$, $k\in [2n-1, {n+1 \choose 2}]$ there is a set $A\in \N$, $|A|=n$ and $|A+A|=k$? For this (undergraduate) question the answer is yes.

The second question is due to Erd\H os and Szemer\'edi (see in \cite{ESZ}): It is easy to see that if $A\subseteq \N$ then for the cardinality of the sums of the subset we have ${|A|+1\choose 2}\leq |P(A)|\leq 2^{|A|}$. They asked: is it true that for every $t$, ${n+1\choose 2}\leq t\leq 2^{n}$ there is a set of integers $A$ with $|A|=n$ and $|P(A)|=t$? In [H96] I gave an affirmative answer.

A similar question on additive energy would be the following. Write $Set_G(n):=\{E(A): \ A\subseteq G; \ |A|=n \}$, where $G$ is any additive structure.

While in the previous examples the possible values of $A+A$ and $P(A)$ were intervals one can guess that the set $Set_G(n)$ is not one. For example let $G=\bo$ and $A=\{0,1\}^k\subseteq \bo$. Then it is easy to see, that the value of $E(A)$ is $6^k$.

Note that in \cite{KT} the authors have shown that if $A\subseteq \bo$ then $E(A)\leq |A|^\varrho$, where $\varrho=\log_26$, and the exponent cannot be replaced by any smaller quantity.  

\subsection{Integer case} First we are going to investigate the case, when $G=\Z$.

\begin{thm}\label{2.1}
Let $\lfloor \frac{n}{3}\rfloor =k$,  $n\equiv r \pmod 4$. Let $$\mathcal{I}:=[2n^2-n+66, k^2+\frac{(k-1)k(2k-1)}{3}-66].$$
Then $\mathcal{I}\cap Set_\Z(n)$ is an arithmetic progression in the form $\{4k+r\}$.
Moreover $\mathcal{J}:=Set_\Z(n)\cap (k^2+\frac{(k-1)k(2k-1)}{3}, \frac{n(n+1)2n+1)}{3}]$ contains $\Omega (n^2)$ elements.
\end{thm}

\begin{proof}

Let us start by some easy observations. Let us note that the additive energy is invariant to the affine transformation, i.e. for every finite set $B$ and integers $a\neq 0, b$ $E(B)=E(aB+b)$.

It is not too hard to show that the parity of $A\subseteq \N$ and $E(A)$ is the same. In fact we prove that for any set $A\subseteq \N$, $|A|\geq 3$, $|A|\equiv E(A)\pmod 4$. Indeed, $E(A)=\sum_x d_A^2(x)=d^2_A(0)+2\sum_{x\neq 0} d_A^{+2}(x)=|A|^2+2\sum_{x\neq 0} d_A^{+2}(x)$. It is easy to check that for every $|A|=3$, we have $E(A)\equiv 3\pmod 4$. Now let $n\geq 3$, and consider $|A'|=n+1$ with biggest element $a_{n+1}$. Denote by $A=A'\setminus a_{n+1}=\{a_1<a_2<\cdots <a_n\}$ and write $x_j=a_{n+1}-a_j; \ j=1,2,\dots, n$. Write  $E(A')=\sum_xd_{A'}^2(x)=d_{A'}^2(0)+2\sum_xd_{A'}^{+2}(x)$. $d_{A'}(0)=n+1$  so $d^2_{A'}(0)-d^2_{A}(0)=2n+1$. Let $d^+_A(x_j)=t_j\geq 0$. Then  $ |d_{A'}^{+2}(x_j)- d_A^{+2}(x_j)|=2t_j+1$. So $E(A')-E(A)=2n+1+2\sum_{j_1}^n(2t_j+1)=4n+4\sum_{j_1}^nt_j+1$. Hence if $|A|\equiv E(A)\equiv r \pmod 4$, then $|A'|\equiv E(A')\equiv r+1 \pmod 4$ as we wanted.

In the first stage, we move downwards from the maximum energy value. So let $A_0=\{a_i=i; \ i=1,2,\dots n\}$. As we mentioned $E(A_0)=(1+o(1))\frac{2}{3}n^3$. The maximal difference is $n-1$, the minimal is $0$, hence we can write $E(A_0)=|A_0|^2+2\sum_{x=1}^{n-1}d^{+2}_{A_0}(x)$. 

For $1\leq k\leq n-2$ we are going to define  the set $A_0^{(k)}$ as follows: for $i=1,2,\dots n-1$ let $a_i=i$ and let $a_n=n+k$. Write shortly $d_k(x)=d^+_{A_{n,k}}(x)$ and $d_0(x)=d^+_{A_0}(x)$. For $1\leq x\leq k$, $d_k(x)=d_0(x)-1$ since $a_n-a_i\geq k+1$ for $i=1,2\dots n-1$. For $k<x\leq n-1$ $d_k(x)=d_0(x)$ and when $n\leq x\leq n+k-1$ then $d_k(x)=1$ since difference bigger than $n-1$ does not occur in $A_0-A_0$.

So we have
$$
E(A_0)-E(A_{n,k})=2\sum_{x=1}^{n-1}d_0^2(x)-2\sum_{x=1}^{n+k-1}d_k^2(x)=
$$
$$
=2\sum_{x=1}^{n-1}d_0^2(x)-2\Big[\sum_{x=1}^k(d_0(x)-1)^2+\sum_{x=k+1}^{n-1}d_0^2(x)+\sum_{x=n}^{n+k-1} 1\Big]=
$$
$$
=2\sum_{x=1}^k(2d_0(x)-1)-2k=4\sum_{x=1}^k(n-x)-4k=4nk-2k^2-6k.
$$
Finally let $A_1:=\{1,2,\dots, n-1, 10^n\}$. Since every $i$, $1\leq i\leq n-1$, $d(a_n-a_i)$ remains $1$ we have that for every $1\leq k\leq n-2$ the gap between two consecutive values of energies is 
$$
E(A_0^{(k+1)})-E(A_0^{(k)})=(E(A_0)-E(A_0^{(k)}))-(E(A_0)-E(A_0^{(k+1)}))=4n-4k-8.
$$
Continue the previous process to obtain the strictly decreasing sequence  of energy values $\{E(A_1^{(k)}); \ k=1,2 \dots, n-3\}$ and generally for $j=1,2,\dots, m$ the sequence $\{E(A_j^{(k)}); \ k=1,2 \dots, n-j-2\}$, where $m$ will be determine later.
 
So the end of the $m^{th}$ step we have the set $A_{m+1}=\{1,2,\dots, n-m, 10^m,10^{m+1},\dots, 10^n\}$ and similarly, as we have seen in the previous process for $k=1,2,\dots, n-m-2$ we obtain
$$
E(A_{m-1}^{(k+1)})-E(A_{m-1}^{(k)})<4(n-m-k)-7.
$$
(note that the elements $10^m,\dots, 10^n$ do not play a role in the change of energy). 

The argument of this stage shows that  $\mathcal{J}$ contains $\Omega (n^2)$ elements.

In the second stage, we move upwards from this given energy value from appropriate $m$.

\begin{lemma}\label{3.2}
Let $X_0=\{x_i\}^m_{i=1}$ be an $m$ element $10$ lacunary sequence of integers, i.e. for $i=1,2,\dots m-1$, $\frac{x_{i+1}}{x_i}\geq 10$. For $k=1,2, \dots, \lfloor m/3 \rfloor$ let $X_k=(X_0\setminus \{x_{3i}\}^{k}_{i=1})\cup  \{x'_{3i}\}^{k}_{i=1}$, where $x'_{3i}=2x_{3i-1}-x_{3i-2}$. We have $E(X_k)-E(X_{k-1})=4$.
\end{lemma}
\begin{proof}
Write briefly $d^+_{X_k}=d^+$. Since $X$ is a $10$ lacunary sequence thus $d^+(x_{3k}-x_{3k-1})=1$. Let us replace $x_k$ by $x'_{3k}=2x_{3k-1}-x_{3k-2}$. Then the difference $x_{3k-1}-x_{3k}$ occurs twice instead of one, the differences $x_{3k}-x_{3k-1}$ and $x_{3k}-x_{3k-2}$ do not occur. The new difference will be $x'_{3k}-x_{3k-2}$ with $d(x'_{3k}-x_{3k-2})=1$. (The values of the other representation functions do not change, just the length of the differences).

So we have $E(X_k)-E(X_{k-1})=2(d(x'_{3k}-x_{3k-2})^2+d^{+2}(x'_{3k}-x_{3k-1})-d^{+2}(x_{3k-1}-x_{3k-2})-d^{+2}(x_{3k}-x_{3k-1}))=2(1^2+2^2-1^2-1^2-1^2)=4$.
\end{proof}
Now if $m>\frac{3n}{4}$ than we have $4n-4m-7<\frac{4m}{3}-1<4\lfloor \frac{m}{3}\rfloor$. By Lemma \ref{3.2}  we can fill the gaps in $E(A_{m-1}^{(k+1)})-E(A_{m-1}^{(k)})$ by sequences with difference $4$ in the interval $[2n^2-n, \frac{k(k+1)2k+1)}{3}]$.
\end{proof}

\medskip

\section{Case $ A_1\times A_2\times \cdots \times A_n\in G=[M]^n$}

In this section, we address a similar issue to the \ref{2.1} theorem, as well as a density vs. energy result related to one of Kane and Tao's results (see [KT]).

Before our results, we will formulate an argument that we will use in the rest of the paper.
\subsection{On density and additive energy of $A_1\times A_2\times \cdots \times A_n$; $A_i \subseteq G; \ i=1,2,\dots n$.}

Let $G$ be any finite semigroup with $[G|=M$. Throughout the rest of the paper we will use the following argument:

If $a\in G^r$ and $b\in G^t$ are two elements, then  $ab$ means  that $a$ and $b$ are literally contiguous (i.e. the two strings are concatenated), and $ab\in G^{r+t}$.

Now let us assume that the sets $U\subseteq G^t:=G_1$ and $V\subseteq G^r:=G_2$ have been defined  with $|U|=M^{c_1t}$ and $|V|=M^{c_2r}$. Let $W:=\{uv: u\in U, \ v\in V\}\subseteq G^{r+t}$. Clearly $|W|=|U||V|=M^{c_1t+c_2r}$. 

Now we are going to show that if $E(U)=|U|^{td_1}$ and $E(V)=|V|^{rd_2}$ then $E(W)=M^{td_1+rd_2}$. Let us assume that $(a_n,b_n,c_n,d_n)\in U^4$ and $(a'_n,b'_n,c'_n,d'_n)\in V^4$ two four-tuples for which $a_n+b_n=c_n+d_n$ and $a'_n+b'_n=c'_n+d'_n$ hold. Then clearly $(a_na'_n)+(b_nb'_n)=(c_nc'_n)+(d_nd'_n)$ also holds. Conversely assume that $(a_na''_n)+(b_nb''_n)=(c_nc''_n)+(d_nd''_n)$ holds, for some four-tuple $(a_na''_n),(b_nb''_n),(c_nc''_n),(d_nd''_n)\in W^4$. By the definition it implies that $a_n+b_n=c_n+d_n$ and $a''_n+b''_n=c''_n+d''_n$. 
Hence $E(W)=E(U)E(V)=M^{td_1+rd_2}$.

One can see by induction that for every $z$, if  $U_i\subseteq G^{r_i}$, $i=1,2,\dots z$ and $W=\{u_1u_2\dots u_z: \ u_i\in U_i \ i=1,2,\dots, z\}$ then $|W|=\prod^z_{i=1}|U_i|$ and $E(W)=\prod^z_{i=1}E(U_i)$.

\medskip

 \subsection{Inverse question in $[M]^n$}

 In this section we will investigate sets in the form $X=A_1\times A_2\times \dots, \times A_n\subseteq [M]^n$, $|A_1|=|A_2|=\dots =|A_n|=w\leq M$, $1\leq |X|=w^n$. Throughout this section write $Set_{[M]^n}(w)=\{E(X_j):  |X_j|=w; X_j\subseteq [M]^n\}$, in increasing order. Write shortly $E_j:=E(X_j)$. One can guess that $E_{j+1}-E_j$ is not bounded. The following example supports this view: Let $G:=\{0,1,2\}^n$ and let $X\subseteq G$ with $|X|=3^m\cdot 2^k$; $2\leq k+m\leq n$. It implies that there are subscripts $i_1,i_2,\dots, i_k$ and $j_1,j_2,\dots, j_m$, for which $|A_{j_1}|=\dots =|A_{j_m}|=3; \ |A_{i_1}|=\dots =|A_{i_k}|=2$ and the cardinality of the rest (if they exist) is $1$. A simple calculation shows that if $Y\subseteq \Z$, $|Y|=3$ then $E(Y)=15$ or $E(Y)=19$. If $|T|=2$ then $E(T)=6$. So by the argument of the previous section we have that for every $|X|=3^m\cdot 2^k$; $2\leq k+m\leq n$,  $E(X)=15^r19^s6^k$; $r+s=m$. Hence for every $j$ $E_{j+1}-E_j$ is at least $6^k$.

 So instead of $E_{j+1}-E_j$ it is reasonable to investigate $E_{j+1}/E_j$.

\begin{thm}\label{3.2} 
There are $N:=c_1M^3$ many $X_i$ in the form $X_i=A_{i,1}\times A_{i,2}\times \dots, \times A_{i,n}\subseteq [M]^n$ for which every $1\leq i\leq N; \ 1\leq j\leq n$ we have $|A_{i,j}|=w$ and $|X_1|=|X_2|=\dots =|X_N|<M^n$, and  such that for every $1\leq i\leq N$, we have
$$
E(X_i)/E(X_{i+1})\leq \Big(1+\frac{c_2}{w^3}\Big).
$$
\end{thm}
Roughly speaking there is a long sequence $\{E_j\}\subseteq Set_{[M]^n}(w^n)$ for which the ratio of the consecutive elements is close to $1$. The constants $c_1=1/27$ and $c_2=360$ are admissible. 

On the other hand

\begin{thm}\label{a3.2}
Let $Set_{[M]^n}(w^n)=\{E_j:j=1,2,\dots\}$ be the increasing sequence of energies. Then there is an effectively computable constant $C$ depending only on $M$ and $n$, such that 
$$
\min_j E_{j+1}/E_j\geq \Big(1+\frac{1}{(en)^C}\Big).
$$
\end{thm}

\begin{proof}[Proof of Theorem \ref{3.2}]

As we have seen at the proof of Theorem \ref{2.1} for any $w\in N$ $Set_\Z(w)$ contains an arithmetic progression $\mathcal{AP}$ with difference $4$ containing in an interval $[\beta_1w^3,\beta_2w^3]$, where $\beta_1=1/90, \beta_2=2/90$ and $w$ is big enough. Let $L=\lfloor w^3/30\rfloor$, and let $A_1,A_2,\dots ,A_L$ be the sets for which $\{E(A_1),E(A_2),\dots ,E(A_L)\}= \mathcal{AP}$. Now we are in the position to define the sets $X_1,X_2,\dots ,X_L$. Let $X_1=A_1\times A_1\times \cdots \times A_1; \  X_2=A_1\times A_1\times \cdots \times A_2; \ \dots ; X_L=A_1\times A_1\times \cdots \times A_L$. By the argument discussed in subsection $3.1$ we get $|X_1|=|X_2|=\dots =|X_L|=w^n$; and $E(X_i)=E(A_1)^{L-1}E(A_i)$. Hence $E(X_i)/E(X_{i+1})=E(A_{i+1})/E(A_i)=(E(A_i)+4)/E(A_i)\leq (1+360/w^3)$.

\end{proof}

\begin{proof}[Proof of Theorem \ref{a3.2}]
The proof is a simple consequence of the following lemma:
\begin{lemma}
Let $1<b_1<b_2<\cdots <b_t$ be a sequence of integers and let $z_1,z_2,\dots, z_t\in \Z$. We have
$
|b_1^{z_1}b_2^{z_2}\cdots b_t^{z_t}-1|>\frac{1}{(eB)^C}
$
where $B=\max\{|z_1|,|z_2|,\dots, |z_r|\}$ and where $C$ is an effectively computable
constant depending only on $t$ and on $b_1,b_2,\cdots ,b_t$.
\end{lemma}
This lemma is a very special case of a theorem of Baker (see \cite{BA} and \cite{E}).

All energies in $Set_{[M]^n}(w^n)=\{E_j:j=1,2,\dots\}$ can be written in the form $E_j=\prod^t_{i=1}E^{u_i}(A_i)$, for some  $|A_i|=w$, $\sum_i u_i\leq n$. So we have $\min_j E_{j+1}/E_j=\prod^n_{i=1}E^{z_i}(A_i)$ with $|z_i|\leq n$, for some $|A_i|=w$, $i=1,2,\dots n$. Hence the theorem.

\end{proof}

\subsection{Density vs additive energy}

In \cite{KT} (see also a generalization in \cite{DG}) the authors investigated the following interesting question. Let $A\subseteq \bo$ be any set, then what can we say on the maximum of $E(A)$ if the cardinality of $A$ is given? They showed that $E(A)\leq |A|^\varrho$, where $\varrho=\log_26$, and the exponent cannot be replaced by any smaller quantity.  

In this section we ask an opposite direction: For a given finite additive structure $G$ and a parameter $0<\eta <1$ what is 
$$
R_G(\delta):=\max_{A\subseteq G}\{\alpha :|A|=|G|^\alpha; \  E_G(A)= |A|^{2+\delta}\}?
$$
In the next theorem we show that this maximum exists and give its asymptotic value.

Let $G$ be a finite abelian group and ket $S\subseteq G$. $S$ is said to be Sidon set if  for every $s_1+s_2=s_3+s_4; \ s_i\in S$ $\{s_1,s_2\}=\{s_3,s_4\}$ holds. We say that $G$ is S-good, if there is  a Sidon set with $|S|=(1+o(1))\sqrt{|G|}$ (see e.g. \cite{O}). Note that for $x\in S+S$, $r_{S+S}(x)=2$, and so $E(S)=4|S|^2$.
\begin{thm}
Let $G_1,G_2,\dots G_n$ be S-good finite abelian groups, $|G_1|=|G_2|=\dots =|G_n|:=M$ and let $\mathcal{G}=G_1\times G_2\times \dots \times G_n$. Fore every $0<\delta<1$  we have $R_\mathcal{G}(\delta)=\frac{1}{2-\delta}(1+o(1))$.
\end{thm}
\begin{proof} We assume that $n$ is an arbitrary but fixed number and $M$ is large enough.
First we prove that $R_\mathcal{G}(\delta)\leq\frac{1}{2-\delta}$.  Let $|A|=|\mathcal{G}|^\alpha$ and $E(A)=|A|^{2+\delta}$. Let us denote the representation function of $A+A$ by $r(x):=|\{(a,a')\in A^2: x=a+a'\}|$. Clearly $\sum_xr(x)=|A|^2$ since we count all pairs of elements of $A$. Furthermore note that $\sum_xr^2_{A+A}(x)=E(A)$ since the sum counts all quadruples. 

Now by the Cauchy inequality
$$
|\mathcal{G}|^{4\alpha}=|A|^4=\Big(\sum_xr_{A+A}(x)\Big)^2\leq |A+A|\sum_xr^2_{A+A}(x)= 
$$
$$
=|A+A|E(A)\leq |\mathcal{G}|A|^{2+\delta}=|\mathcal{G}|^{1+\alpha(2+\delta)}.
$$
Hence $4\alpha\leq 1+\alpha(2+\delta)$ which gives
$
\alpha\leq \frac{1}{2-\delta}.
$
Now we will complete our theorem showing the bound $(1+o(1))/(2-\delta)$.

Let $1\leq k\leq n$. We are going to define sets $A_k\subseteq \mathcal{G}$ for which $E(A_k)=|A_k|^{2+\delta}$ and $|A_k|=|\mathcal{G}|^{(1+o(1))/(2-\delta)}$.

Let $A_k:=\prod_{i=1}^kS_i\times \prod_{j=k+1}^nG_j$, where for every $i$ and $j$, $|S_i|=(1+o(1))\sqrt{M}$ and clearly $E(G_j)=M^3$, since every $a,b,c\in G_j$ $a+b-c\in G_j$.

We have $|A_k|=(1+o(1))M^{k/2}M^{n-k}=(1+o(1))M^{n-k/2}$ and $E(A_k)=(1+o(1))(4M)^kM^{3(n-k)}=(1+o(1))M^{3n-2k+2k\log 4/\log M}$. 

Thus 
$
E(A_k)=(1+o(1))|A_k|^{\frac{6n-4k}{2n-k}+2k\log 4/\log M}=|A_k|^{2+\frac{2n-2k}{2n-k}(1+o(1))}.
$
So we have $\alpha=\frac{2n-k}{2n}$, and $\delta=\frac{2n-2k}{2n-k}(1+o(1))$. Now
$$
\frac{1}{2-\delta}=\frac{1}{2-\frac{2n-2k}{2n-k}(1+o(1))}=\frac{2n-k}{2n}(1+o(1))=\alpha(1+o(1)).
$$
\end{proof}

{\bf Acknowledgment} This work is supported by grant K-129335.

\end{document}